\documentclass[12pt, reqno]{amsart}
\usepackage{graphicx}
\usepackage{subfigure}
\usepackage{latexsym}
\usepackage{amsmath}
\usepackage{amssymb}

\usepackage[usenames]{color} 
\usepackage{soul} 
\usepackage{pstricks}

\usepackage{amscd}
\usepackage{color}
 \newtheorem{theorem}{Theorem}[section]
 \newtheorem{Def}[theorem]{Definition}
 \newtheorem{Prop}[theorem]{Proposition}
 \newtheorem{Lem}[theorem]{Lemma}
 \newtheorem{Cor}[theorem]{Corollary}
 \newtheorem{Rem}[theorem]{Remark}

 \newtheorem{Exa}[theorem]{Example}

\def\ba{{\bf a}}
\def\bb{{\bf b}}
\def\bc{{\bf c}}

\def\b1{{\bf 1}}

\def\bt{{\bf t}}
\def\bu{{\bf u}}

\newcommand{\D}{{\mathcal D}}
\newcommand{\E}{{\mathcal E}}
\newcommand{\F}{{\mathcal F}}

\date {}

\begin{document}

\title{Self-similar sets, simple augmented trees, and their Lipschitz equivalence}

\author{Jun Jason Luo}
\address{College of Mathematics and Statistics, Chongqing University,  401331 Chongqing, China
\newline\indent Institut f\"ur Mathematik, Friedrich-Schiller-Universit\"at Jena, 07743 Jena, Germany}
\email{jasonluojun@gmail.com; jun.luo@uni-jena.de}

\subjclass[2010]{Primary 28A80; Secondary 05C05, 20F65}

\begin{abstract}
Given an iterated function system (IFS) of contractive similitudes, the theory of Gromov hyperbolic graph on the IFS has been established recently. In the paper, we introduce a notion of simple augmented tree which is a Gromov hyperbolic graph. By generalizing a combinatorial device of rearrangeable matrix, we show that there exists a near-isometry between the simple augmented tree and the symbolic space of the IFS, so that their hyperbolic boundaries are Lipschitz equivalent. We then apply this to consider the Lipschitz equivalence of self-similar sets with or without assuming the open set condition. Moreover, we also provide a criterion for a self-similar set to be a Cantor-type set which completely answers an open question raised in \cite{LaLu13}.  Our study extends the previous works.
\end{abstract}

\footnote{The research is supported in part by the NNSF of China (No.11301322), the Fundamental and Frontier Research Project of Chongqing (No.cstc2015jcyjA00035)}

\keywords{self-similar set, simple augmented tree, quotient space, hyperbolic boundary, rearrangeable matrix, open set condition, weak separation condition}

\maketitle

\section{Introduction}
A self-similar set $K$ is defined as the limit set of an  iterated function system (IFS) of contractive similitudes, and the iteration is addressed by a  symbolic space of finite words which forms a tree. The topological boundary of the tree is a Cantor-type set. However the tree does not capture all the geometric and analytic properties of $K$. By incorporating more information of $K$ onto the tree,  Kaimanovich \cite{Ka03} first explored the concept of ``augmented tree'' on the Sierpinski gasket such that it is a hyperbolic graph in the sense of Gromov (\cite{Gr87}\cite{Wo00}) and its hyperbolic boundary is homeomorphic to the gasket. Lau and Wang (\cite{LaWa09}\cite{Wa14}) extended this idea on more general self-similar sets satisfying the open set condition (OSC) or the weak separation condition. Recently they \cite{LaWa16} completed the previous studies by removing some superfluous conditions, and obtained that for any IFS, the augmented tree is always a hyperbolic graph. Moreover, the hyperbolic boundary is  H\"older equivalent to the $K$.  The setup of augmented trees  connects fractal geometry, graph theory and Markov chains, it  has been frequently used to study the random walks on self-similar sets and the induced Dirichlet forms  (\cite{Ki10}\cite{JuLaWa12}\cite{LaWa15}\cite{KoLaWo}\cite{KoLa}). On the other hand, the author and his coworkers made a first attempt to apply augmented trees to the study on Lipschitz equivalence of self-similar sets.  In a series of papers (\cite{LaLu13}\cite{DLL15}\cite{L13}),  we investigated in detail  the Lipschitz equivalence of totally disconnected self-similar sets with equal contraction ratio and their hyperbolic boundaries.

Recall that two metric spaces $(X, d_1)$ and $(Y, d_2)$ are said to be {\it Lipschitz equivalent}, write $X\simeq Y$, if there exists a bi-Lipschitz map $\sigma: X \to Y$, i.e., $\sigma$ is a bijection and there exists a constant $C>0$ such that
$$C^{-1}d_1(x,y)\le d_2(\sigma(x), \sigma(y))\le C d_1(x,y), \quad\text{for all } x,y\in X.$$

The Lipschitz equivalence of Cantor sets was first considered in \cite{CoPi88} and \cite{FaMa92}. For its extension on self-similar sets, it has been undergoing rapid development recently (\cite{DaSe97}\cite{RaRuXi06}\cite{XiRu07}\cite{MaSa09}\cite{XiXi10}\cite{RaRuWa10}\cite{LlMa10}\cite{DeHe12}\cite{RuWaXi12}\cite{RaZh15}\cite{XiXi12}\cite{XiXi13}). However, most of the studies are based on the nice geometric structure of self-similar sets such as Cantor sets or totally disconnected self-similar sets with OSC. There are few results on the non-totally disconnected self-similar sets (\cite{LuLi16}\cite{RuWa16}) or self-similar sets without OSC.

In this paper, we unify our previous approaches on augmented trees and consider the Lipschitz equivalence of more general self-similar sets which allow non-equal contraction ratios and substantial overlaps.

Let $\{S_i\}_{i=1}^N$ be an IFS  on ${\mathbb{R}}^d$ where $S_i$ has a contraction ratio $r_i\in (0,1)$, let $K$ be the self-similar set of the IFS satisfying $K={\bigcup}_{i=1}^N S_i(K)$. Let $\Sigma =\{1, \dots, N\}$ and ${\Sigma^*}=\bigcup_{n=0}^{\infty} \Sigma^n$ be the symbolic space (by convention, $\Sigma^0 = \emptyset$). For $x = i_1\cdots i_k\in \Sigma^*$, we denote by $S_x=S_{i_1}\circ   \cdots   \circ S_{i_k}$, and $r_x= r_{i_1} \cdots r_{i_k}$. Let $r=\min\{r_i: i=1,2,\dots, N\}$.  We also define a new symbolic space
\begin{equation*}
{X}_n = \{x = i_1\cdots i_k\in \Sigma^*:\  r_x \leq r^n < r_{i_1} \cdots r_{i_{k-1}}\} \quad \text{and}\quad X=\bigcup_{n=0}^\infty X_n.
\end{equation*}
In the special case that all $r_i$ are equal, $X=\Sigma^*$. In general, $X$ is a proper subset of $\Sigma^*$.  If $x \in {X}_n$, we denote the length by $|x|=n$, and say that $x$ lies in level $n$. The $X$ has a natural tree structure according to the standard concatenation of finite words, we denote the edge set by $\E_v$ ($v$ for vertical). We also define a horizontal edge for a pair $(x,y)$ in $X\times X$ if $x, y\in X_n$ and $\text{dist}(S_x(K), S_y(K))\le \kappa r^{|x|}$ where $\kappa>0$ is a fixed constant, and denote this set of edges by $\E_h$ ($h$ for horizontal). Let $\E=\E_v\cup \E_h$, then the graph $(X, \E)$ is an augmented tree in the sense of Kaimanovich (\cite{Ka03}\cite{LaWa16}). Lau and Wang \cite{LaWa16} already showed that the augmented tree $(X, \E)$  is a Gromov hyperbolic graph and its hyperbolic boundary $\partial X$ under a visual metric is H\"older equivalent to the $K$.

Based on this, our approach to the Lipschitz equivalence of self-similar sets is to lift the consideration to the augmented tree $(X,\E)$. We define a {\it horizontal component} of $X$ to be the maximal connected horizontal subgraph $T$ in some level with respect to $\E_h$. Let $\mathcal F$ be the set of all horizontal components of $X$. For $T\in{\mathcal F}$, we use $T_{\mathcal D}$ to denote the union of $T$ and its all descendants, with the subgraph structure inherited from $(X, \E)$. We say that $T, T'\in {\mathcal F}$ are equivalent if $T_{\mathcal D}$ and $T'_{\mathcal D}$ are graph isomorphic. We call $(X, \E)$ {\it simple} if there are only finitely many equivalence classes in $\mathcal F$.  Similarly, we can define a simple tree for $(X,\E_v)$ if there are only finitely many equivalence classes of vertices. Obviously, if $(X, \E)$ is simple then $(X, \E_v)$ is simple. We use $A$ (resp. $B$) to denote the incidence matrix of $(X, \E)$ (resp. $(X, \E_v)$), which encodes graph relation of the equivalence classes. Let a vector $\bu$ represent the finite classes in ${\mathcal F}$ by means of the classes of vertices. (see Definition \ref{def of simple tree} and Remark \ref{B-u-remark}.) The following is our first main result

\begin{theorem}\label{th-main}
Suppose an augmented tree $(X, {\mathcal E})$ is simple, and suppose the incidence matrix $A$ is $(B,\bu)$-rearrangeable. Then there exists a near-isometry between the augmented tree $(X, {\mathcal E})$ and the tree $(X,{\mathcal E}_v)$, so that $\partial (X,\E) \simeq \partial (X,\E_v)$.
\end{theorem}

The concept of rearrangeable matrix was initiated by the author (\cite{LaLu13}\cite{DLL15}\cite{L13}) in studying the Lipschitz equivalence.  A lot of nonnegative integer matrices including primitive matrices are rearrangeable. It  is a combinatorial device to arrange the vertices and edges of the augmented tree properly  to construct the near-isometry, which is stronger than the rough isometry in literature. In the paper, we extend the original definition of rearrangeable matrix  so as to solve more general situations (please see Definition \ref{def rearrangeable}).

For an IFS with substantial overlaps (i.e., without the OSC), it may happen that $S_x=S_y$ for $x\ne y$. In this case, the augmented tree $(X, \E)$ may be no longer simple. But we can modify the augmented tree by identifying $x, y\in X$ for $|x|=|y|$ and $S_x=S_y$, and let $(X^\sim, \E)$ denote the quotient space with the induced graph. Moreover, we can further reduce the graph $(X^\sim, \E_v)$ into a tree $(X^\sim, \E^*_v)$ (see Section 2).  Then following the same proof of Theorem \ref{th-main}, we have

\begin{theorem}\label{th-quotient}
Suppose the quotient space $(X^\sim, {\mathcal E})$ of  an augmented tree  is simple, and suppose the incidence matrix $A$ is $(B,\bu)$-rearrangeable. Then there exists a near-isometry between the quotient space $(X^\sim, {\mathcal E})$ and the reduced tree $(X^\sim, {\mathcal E}^*_v)$, so that $\partial (X^\sim, {\mathcal E}) \simeq \partial (X^\sim, {\mathcal E}^*_v)$.
\end{theorem}

An IFS is called {\it strongly separated} \cite{F} if $S_i(K)\cap S_j(K)=\emptyset$ for $i\ne j$, in this case  the $K$ is called {\it dust-like} (\cite{FaMa92}\cite{LaLu13}). By reducing the Lipschitz equivalence on the augmented trees in Theorems \ref{th-main} and \ref{th-quotient} to the self-similar sets, we have our second main result

\begin{theorem}\label{th-main2}
Under the assumption of Theorem \ref{th-main}, the associated self-similar set $K$ is Lipschitz equivalent to a dusk-like self-similar set with the same contraction ratios as for $K$.

Under the assumption of Theorem \ref{th-quotient}, the associated self-similar set $K$ is Lipschitz equivalent to a Cantor-type set (may be not a dust-like self-similar set).
\end{theorem}

The first part of Theorem \ref{th-main2} is applied for the self-similar sets satisfying the OSC; while the second part is applied for the self-similar sets satisfying the weak separation condition (WSC) \cite{LaNg99}. The WSC contains many important overlapping cases, it has been studied extensively in connection with the multifractal structure of self-similar measures (see \cite{LaNg99}\cite{FL09}\cite{DeLaNg13} and the references therein).

Finally, we provide a criterion for the hyperbolic boundary (or the associated self-similar set) to be totally disconnected (i.e., a Cantor-type set) from the perspective of augmented trees which also answers  an open question raised in \cite{LaLu13}.

\begin{theorem}\label{th-main3}
Let $(X, {\mathcal E})$ be an augmented tree of bounded degree. Then the hyperbolic boundary $\partial X$ (or the self-similar set $K$) is totally disconnected if and only if the sizes of horizontal components are uniformly bounded.

The theorem is still valid if we replace $X$ by its quotient space $X^\sim$.
\end{theorem}

The paper is organized as follows: In Section 2, we recall basic results on hyperbolic graphs and introduce the notion of  simple augmented tree;  moreover, we state Theorem \ref{th-main}  and prove Theorem \ref{th-quotient} there. In Section 3, we generalize the concept of rearrangeable matrix and prove Theorem \ref{th-main}.  In Section 4, we prove Theorem \ref{th-main2} and give two examples to illustrate our results. Theorem \ref{th-main3} will be showed in Section 5. We include some remarks and open questions in Section 6.

\section{Simple augmented trees}
In this paper,  an infinite connected graph $(X,\E)$ will be considered, where $X$ is a set of vertices which is countably infinite, $\E$ is a set of edges which is a symmetric subset  of $(X \times X) \setminus \{(x,x) : x \in X\}$.   By a {\it path} in $(X,\E)$ from $x$ to $y$, we mean a finite sequence $x =x_0, x_1, \dots , x_n=y$ such that $(x_i,x_{i+1})\in {\mathcal E}, i=0, \dots, n-1$. We call $X$ a {\it tree} if the path between any two distinct vertices is unique. We equip a graph $X$ with an integer-valued metric $d(x,y)$, which is the length of a geodesic, denoted by $\pi(x,y)$, from $x$ to $y$. Let $\vartheta \in X$ be a reference point in $X$ and call it the {\it root} of the graph. We use $|x|$ to denote $d(\vartheta,x),$ and say that $x$ belongs to the $n$-th level of the graph if $d(\vartheta,x)=n$.

Recall that a graph $X$ is called {\it hyperbolic} (with respect to $\vartheta$) if there is $\delta >0$ such that
$$|x \wedge y| \geq \min\{|x \wedge z|,  |z\wedge y|\}-\delta \quad \text{for any }   x,y,z \in  X,$$
where $|x \wedge y| :=\frac{1}{2}(|x|+|y|-d(x,y))$ is the {\it Gromov product} (\cite{Gr87}\cite{Wo00}). We choose $a>0$, and define for $x,y\in X$ and a root $\vartheta$
\begin{equation}\label{eq2.1}
{\rho_a}(x,y)=\left\{
\begin{array}{ll}
0 & x=y \\
\exp(-a |x \wedge y| ) &  x\ne y.
\end{array} \right.
\end{equation}
This is not necessarily a metric unless $X$ is a tree. However $\rho_a$ is equivalent to a metric as long as $a$ is small enough (i.e., $\exp(3\delta a)<\sqrt{2}$ \cite{Wo00}). Hence we regard $\rho_a$ as a {\it visual metric} for convenience. The {\it hyperbolic boundary} of $X$ is defined by $\partial X := \overline{X}\setminus X$, where $\overline{X}$ is the  $\rho_a$-completion of $X$.

It is often useful to identify $\xi \in \partial X$ with the {\it geodesic rays} in $X$ that converge to $\xi$, i.e., an infinite path $\pi[x_1,x_2,\dots]$ such that $(x_i, x_{i+1})\in \E$ and any finite segment of the path is a geodesic. It is known that two geodesic rays $\pi[x_1,x_2,\dots], \pi[y_1,y_2,\dots]$ are equivalent as Cauchy sequence  if and only if $d(x_n, y_n)\le c_0$ for all but finitely many $n$, where $c_0>0$ is independent of the rays \cite{Wo00}.

Both Gromov product and the visual metric $\rho_a$ can be extended onto the completion $\overline{X}$ \cite{BH99}, and the hyperbolic boundary  $\partial X$ is a compact subset of $\overline{X}$. We mention that a tree is always a hyperbolic graph with $\delta=0$ and its hyperbolic boundary is a Cantor-type set.

Throughout the paper, we let  $\{S_i\}_{i=1}^N$  be an IFS of contractive similitudes on ${\mathbb{R}}^d$ with contraction ratios $r_i\in (0,1), i=1,\dots, N$. Then there exists a unique non-empty compact subset $K$ of ${\mathbb{R}}^d$ \cite{Hu81} such that
\begin{equation}\label{set equ}
K=\bigcup_{i=1}^N S_i(K).
\end{equation}
We call such $K$ a {\it self-similar set}, and call $K$ {\it dust-like}  if it satisfies $S_i(K) \cap S_j(K)  = \emptyset $ for $i \not = j$. A dust-like self-similar set must be totally disconnected, but the converse is not true.

Let $\Sigma =\{1, \dots, N\}$ and let ${\Sigma^*}=\bigcup_{n=0}^{\infty} \Sigma^n$ be the symbolic space representing the IFS (by convention,
$\Sigma^0 = \emptyset$). For $x = i_1\cdots i_k$, we denote by $S_x=S_{i_1}\circ   \cdots   \circ S_{i_k}$, $K_x:= S_x(K)$ and $r_x= r_{i_1} \cdots r_{i_k}$. Let $r=\min\{r_i: i=1,2,\dots, N\},  X_0=\emptyset$, we define a new symbolic space
\begin{equation*}
{X}_n = \{x = i_1\cdots i_k\in \Sigma^*:\  r_x \leq r^n < r_{i_1} \cdots r_{i_{k-1}}\} \quad \text{and}\quad X=\bigcup_{n=0}^\infty X_n.
\end{equation*}
In general $X$ is a proper subset of $\Sigma^*$. If $x = i_1\cdots i_k\in {X}_n$, we denote the length by $|x|=n$, and say that $x$ is in level $n$ (here $|x|$ is not the real length of the word $x$ in  $\Sigma^*$).

Note that for each $y\in X_n, n\ge 1$ there exists a unique $x\in X_{n-1}$ and $z\in \Sigma^*$ such that $y=xz$ in the standard concatenation. This defines a natural tree structure on $X$: $x,y$ are said to be joined by a vertical edge denoted by $(x,y)$ (or $(y,x)$). Let $\E_v$ be the set of all vertical edges, then $(X,\E_v)$ is a tree with root $\vartheta=\emptyset$. For $y\in X$, let  $y^{-1}$ denote $x\in X$ such that $(x,y)\in {\mathcal E}_v$ and  $|y|=|x|+1$, we call such $x$ an ancestor of $y$.  Define inductively  $y^{-n}=(y^{-(n-1)})^{-1}$ to be the $n$-th generation ancestor of $y$.

Given a tree $(X,\E_v)$, we can add  more (horizontal)  edges according to \cite{LaWa16}:
$${\mathcal E}_h=\{(x,y)\in X\times X: |x|=|y|, x\ne y \ \text{and}\  \text{dist}(S_x(K), S_y(K))\le \kappa r^{|x|}\}.$$
Trivially, $\E_h$ is a symmetric subset of $X\times X$ and satisfies:
$$
(x,y)\in {\mathcal E}_h \quad \Rightarrow \quad |x|=|y| \ \ \text{and}\ \ \text{either}~ x^{-1}=y^{-1} ~\text{or}~ (x^{-1},y^{-1})\in {\mathcal E}_h.
$$
Let ${\mathcal E}={\mathcal E}_v\cup {\mathcal E}_h$, then the graph $(X,{\mathcal E})$ is called an {\it augmented tree} in the sense of Kaimanovich \cite{Ka03}.

The above horizontal edge set $\E_h$ is modified from the previous works (\cite{Ka03}\cite{LaWa09}\cite{Wa14}) where the last condition was $S_x(K)\cap S_y(K)\ne\emptyset$ (i.e., $\kappa=0$). The advantage of this new definition of $\E_h$ is that it allows us to eliminate some superfluous conditions added before.

We say that a geodesic $\pi(x,y)$ is a {\it horizontal geodesic} if it consists of horizontal edges only; it is called a {\it canonical geodesic}  if there exist $u,v\in \pi(x,y)$ such that $\pi(x,y)=\pi(x,u)\cup\pi(u,v)\cup\pi(v,y)$ with $\pi(u,v)$ a horizontal path and $\pi(x,u),\  \pi(v,y)$ vertical paths. Moreover,  the horizontal part of the canonical geodesic  is required to be on the highest level (i.e., for any geodesic $\pi^{\prime}(x,y)$, $\text{dist}(\vartheta,\pi(x,y))\leq \text{dist}(\vartheta,\pi^{\prime}(x,y))$). Then the Gromov product of an augmented tree can be expressed by
\begin{equation}\label{eq.cano.geo.identity}
|x\wedge y| = l-h/2
\end{equation}  where $l$ and $h$ are the level and the length of the horizontal part of the canonical geodesic $\pi(x,y)$.

\begin{theorem}[\cite{LaWa09}]{\label{th2.3}}
An augmented tree $(X, {\mathcal E})$ is hyperbolic if and only if there exists a constant $c>0$ such that any  horizontal  geodesic is bounded by $c$.
\end{theorem}

By using the above criterion for hyperbolicity, Lau and Wang \cite{LaWa16} showed that the augmented tree $(X, \E)$ is always hyperbolic.  On an augmented tree, for $T\subset {X_n}$, we use $T_{\mathcal D}$ to denote the union of $T$ and its all descendants, that is,
$$T_{\mathcal D}=\{x\in X:\  x|_n\in T\}$$
where $x|_n$ is the initial part of $x$ with length $n$. If $T$ is connected, then $T_{\mathcal D}$, equipped with the edge set $\E$ restricted on $T_{\mathcal D}$, is a subgraph of $X$. Moreover, if $(X,{\mathcal E})$ is hyperbolic, then $T_{\mathcal D}$ is also hyperbolic. We say  that $T$ is an {\it $X_n$-horizontal component} if \ $T\subset {X_n}$ is a maximal connected subset with respect to ${\E}_h$.  We let $\F_n$ denote the family of all  $X_n$-horizontal components, and let $\F=\cup_{n\ge 0}\F_n$.

For $T, T'\in\F$, we  say  that $T$ and $T'$ are {\it equivalent}, denote by $T \sim T'$, if there exists a graph isomorphism $g:\, T_\D\to T'_\D$, i.e., the map $g$ and the inverse map $g^{-1}$ preserve the vertical and horizontal edges of $T_\D$ and $T'_\D$. We denote  by $[T]$ the equivalence class containing $T$.

\begin{Def} [\cite{DLL15}] \label{def of simple tree}
We call an augmented tree $(X, \E)$  {\it simple} if the equivalence classes in $\F$ is finite, i.e., ${\mathcal F}/\sim$ is finite. Let $[T_1],\dots,[T_m]$ be the equivalence classes, and let $a_{ij}$  denote the cardinality of the horizontal components of offspring of $T \in [T_i]$ that belong to $[T_j]$.  We call $A =[a_{ij}]_{m\times m}$ the incidence matrix of $(X,\E)$.
\end{Def}

\begin{Rem}\label{B-u-remark}
Similarly, for the tree $(X, \E_v)$, we say that two vertices $x,y$ of $X$ are {\it equivalent}, denote by $x\sim y$,  if  the two subtrees $\{x\}_{\D}$ and $\{y\}_{\D}$ (defined as $T_{\mathcal D}$ above) are graph isomorphic. We call the tree $(X, \E_v)$ {\it simple} if  $X$ has only finitely many non-equivalent vertices, say $[\bt_1],\dots, [\bt_n]$ (For convenience we always assume the root $\vartheta\in [\bt_1]$). In this case, we denote the incidence matrix by $B=[b_{ij}]_{n\times n}$, where $b_{ij}$  means the cardinality of offspring of a vertex $x\in [\bt_i]$ that belong to $[\bt_j]$.  Obviously, if an augmented tree $(X, \E)$ is simple then the associated tree $(X, \E_v)$ is also simple. Therefore any horizontal component $T$ can be represented by the equivalence classes of vertices $[\bt_1],\dots, [\bt_n]$.  More precisely, for $T\in [T_i]$ where $1\le i\le m$, suppose $T$ consists of $u_{ij}$ vertices belonging to $[\bt_j], j=1,\dots, n$, then we represent it by using a vector $\bu_i=[u_{i1},\dots, u_{in}]$ (relative to $[\bt_1],\dots, [\bt_n]$), and let $\bu=[\bu_1,\dots, \bu_m]$. Then $\bu$ is a representation of the classes $[T_1],\dots,[T_m]$ with respect to the classes $[\bt_1],\dots, [\bt_n]$.
\end{Rem}

\begin{Prop}\label{th.simple tree}
Let $(X, \E_v)$ and $(Y,\E'_v)$ be two simple trees defined as above. If they have the same incidence matrix $B$, then they are graph isomorphic.
\end{Prop}

\begin{proof}
Let $X=\bigcup_{k=0}^\infty X_k, Y=\bigcup_{k=0}^\infty Y_k$ where $X_0=\{\vartheta\}, Y_0=\{\vartheta'\}$ are their roots. Let $B=[b_{ij}]_{n\times n}$ be the incidence matrix associated to the equivalence classes of vertices $[\bt_1],\dots,[\bt_n]$ of $X$ (or the equivalence classes of vertices $[\bt_1'],\dots,[\bt_n']$ of $Y$ respectively). Write $B^k=[b^k_{ij}]_{n\times n}$ for $k\ge 1$. Since $\vartheta\in [\bt_1]$, by the definition of incidence matrix, $X_k$ consists of  $b^k_{1j}$ vertices belonging to $[\bt_j]$ where $1\le j\le n$. Similarly, $Y_k$ consists of  $b^k_{1j}$ vertices belonging to $[\bt_j']$ where $1\le j\le n$ as $\vartheta'\in [\bt_1']$. Hence we can define a bijective map $\sigma: X\to Y$ such that if $x\in X_k$, then $\sigma(x)\in Y_k$; and if $x\in [\bt_j]$ then $\sigma(x)\in [\bt_j']$. That $\sigma$ is the desired graph isomorphism.
\end{proof}

\begin{Def}
Let $X,Y$ be two connected graphs. We say  $\sigma: X  \to Y$ is a  near-isometry if it is a bijection and there exists $L>0$ such that for any $x,y\in X$,
$$
\big| d(\sigma(x),\sigma(y))-d(x,y)\big|<L.
$$
\end{Def}

\begin{Lem}[\cite{LaLu13}]
Let $X$, $Y$ be two hyperbolic graphs that are equipped with the visual metrics with the same parameter  $a$ (as in (\ref{eq2.1})). Suppose there exists a near-isometry  $\sigma : X \to Y$, then \ $\partial X \simeq \partial Y$.
\end{Lem}

Now we state our main theorem as below.

\begin{theorem}\label{th-main-part1}
Suppose an augmented tree $(X, {\mathcal E})$ is simple, and suppose the incidence matrix $A$ is $(B,\bu)$-rearrangeable (where $B$ and $\bu$ are defined as in Remark \ref{B-u-remark}). Then there exists a near-isometry between the augmented tree $(X, {\mathcal E})$ and the tree $(X,{\mathcal E}_v)$, so that $\partial (X,\E) \simeq \partial (X,\E_v)$.
\end{theorem}

The concept of {\it rearrangeable matrix} is a generalization of the one in \cite{LaLu13}, which is a crucial device to construct the near-isometry. Since the definition is a little complicated, we will introduce it in detail and prove the theorem in the next section.

Recall that a graph $(X,\E)$ is of {\it bounded degree} if $\max\{\text{deg}(x): x\in X\}<\infty$ where $\text{deg}(x)=\#\{y\in X: (x,y)\in{\mathcal E}\}$ is the total number of edges connecting $x$. Lau and Wang \cite{LaWa16} proved that {\it the  augmented tree $(X, {\mathcal E})$ is of bounded degree if and only if the IFS satisfies the OSC.}  If the augmented tree $(X,\E)$ is simple then it  is of bounded degree, hence the IFS satisfies the OSC.

However, for an overlapping IFS that does not satisfy the OSC, it may happen that $S_x=S_y$ for $x\ne y\in X$. Then the augmented tree may be no longer simple. In this situation, we need to modify the augmented tree by identifying $x,y$ for $|x|=|y|$ and $S_x=S_y$, and  define a quotient space $X^\sim$ of $X$ by the equivalence relation. Let $\E_v, \E_h$ be the sets of vertical edges and horizontal edges inherited from the original ones  and let $(X^\sim, \E)$ denote the quotient space with the induced graph. Moreover, it was also proved in  \cite{LaWa16} that {\it the graph $(X^\sim, {\mathcal E})$ is of bounded degree if and only if the IFS satisfies the weak separation condition (WSC)}.  The concept of the WSC was first introduced in \cite{LaNg99} to study multifractal measures. It is well-known that the WSC is strictly weaker than the OSC and the finite type condition \cite{NgWa01}. We refer to a survey paper \cite{DeLaNg13} for the detailed discussion on various separation conditions in fractal geometry.

Under the quotient space $(X^\sim, \E)$,  we can still define the equivalence classes of horizontal components and vertices and the concept of ``simple'' as above. But it should be more careful that we can not reduce the quotient space $(X^\sim, \E)$ into a tree by only removing the horizontal edges, that is, $(X^\sim, \E_v)$ will not be a tree. In that situation, one vertex of $X^\sim$ may have multi parents, that destroys the tree structure. However, we can further reduce the  graph $(X^\sim, \E_v)$ into a real tree $(X^\sim, \E^*_v)$ in the following way: for a vertex $y\in X^\sim$, let $x_1,\dots, x_n$ be all the parents of $y$ such that $(x_i, y)\in \E_v$ and $|y|=|x_i|+1$ where $i=1,\dots, n$.   Suppose $x_1<x_2<\cdots<x_n$ in the lexicographical order for the finite words, then we keep only the edge $(x_1, y)$ and remove all other vertical edges. We denote by $\E^*_v$ the set of remained edges. Obviously $(X^\sim, \E^*_v)$  indeed is a tree. Moreover, if the quotient space $(X^\sim, \E)$ is simple then the reduced tree $(X^\sim, \E^*_v)$ is also simple.

\begin{Exa}\label{example-quotienttree}
{\rm  Let $S_1(x)=\frac{1}{4}x, S_2(x)=\frac14(x+\frac34), S_3(x)=\frac14(x+\frac74), S_4(x)=\frac14(x+3)$ be an IFS. As $S_{14}=S_{21}$, the IFS satisfies the WSC, but not the OSC (see Figure \ref{quotienttree} (a)). Hence the augmented tree $(X, \E)$ is not simple. But we can consider the quotient space $(X^\sim, \E)$ by identifying vertices, then $x=\{14, 21\}$, as an equivalence class, is a vertex in $X^\sim$. Let $T_1=\{\vartheta\}, T_2=\{1,2,3\}$ and $T_3=\{24,31,32,33\}$ be the horizontal components of the quotient space $(X^\sim, \E)$. We  shall show that there are only three equivalence classes:  $[T_1], [T_2], [T_3]$ by Lemma 5.1 of \cite{LaLu13}, hence $(X^\sim, \E)$ is simple, and the incidence matrix is
$$A=\left[\begin{array}{rrr}
          1 &  1 & 0 \\
          1 & 2 & 1 \\
          1 & 2 & 2
    \end{array} \right].$$
Indeed, we can find out all the equivalence classes  in the first three iterations as follows: In the first iteration, $\{4\}\sim T_1$; in the second iteration, $\{34\}\sim T_1$ and $\{11,12,13\}, \{x,22,23\}, \{41,42, 43\}$ are all equivalent to $T_2$; in the third iteration, $\{334\}\sim T_1$, $\{241, 242, 243\}, \{321, 322, 323\}$ are equivalent to $T_2$, $\{244,311,312, 313\}$, $\{324, 331,332,333\}$ are equivalent to $T_3$.

On the other hand, there are two different paths $[\vartheta, 1, x]$ and $[\vartheta, 2, x]$ joining $\vartheta$ and $x$ in $X^\sim$, hence $(X^\sim, \E_v)$ is not a tree (see Figure \ref{quotienttree} (c)). However, we can reduce $(X^\sim, \E_v)$ into a tree $(X^\sim, \E^*_v)$ by using the method in the previous paragraph, for example removing the vertical edge $(2,x)$ as $1<2$ (see Figure \ref{quotienttree} (d)). It is not hard to check that the equivalence classes of vertices of $(X^\sim, \E^*_v)$ are $[\{\vartheta\}], [\{2\}]$. Therefore $(X^\sim, \E^*_v)$  is simple and the incidence matrix is
$$B=\left[\begin{array}{rr}
          3 &  1 \\
          2 & 1
    \end{array} \right].$$
}
\begin{figure} [h]
    \centering
     \subfigure[]{
    \includegraphics[width=6.5cm]{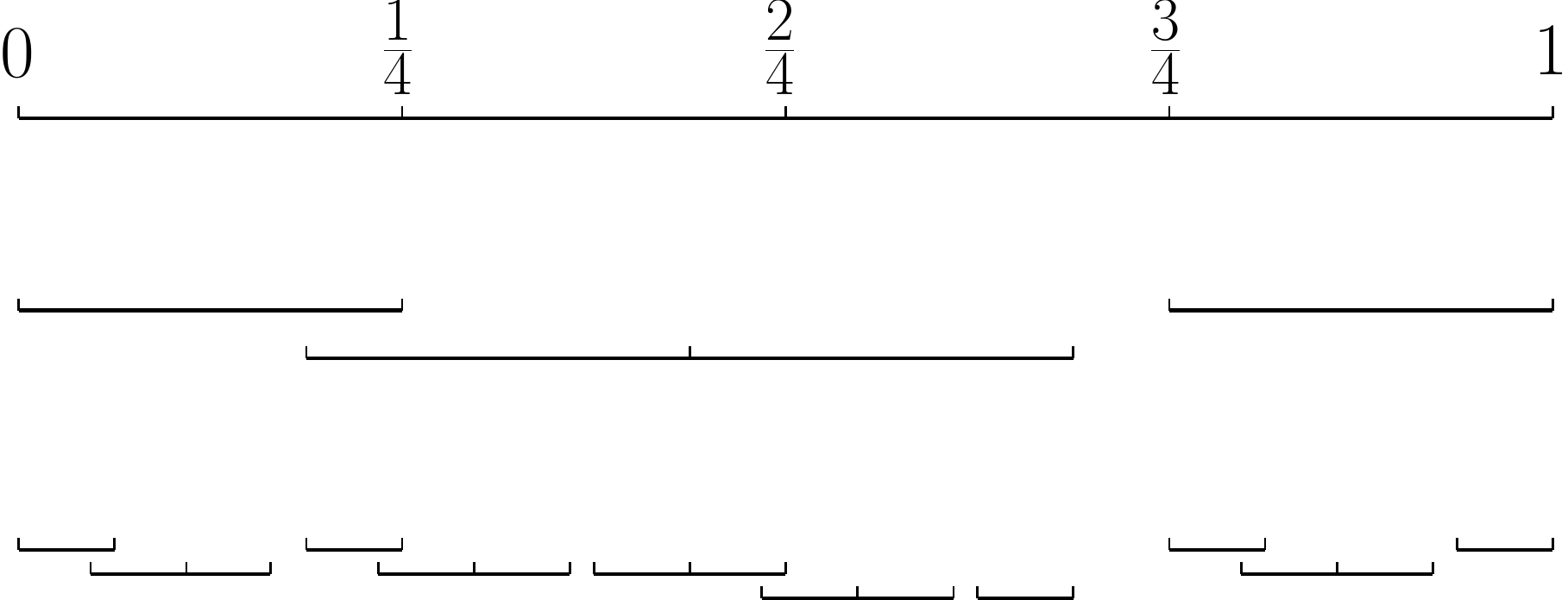}
    }\qquad
     \subfigure[]{
    \includegraphics[width=6.5cm]{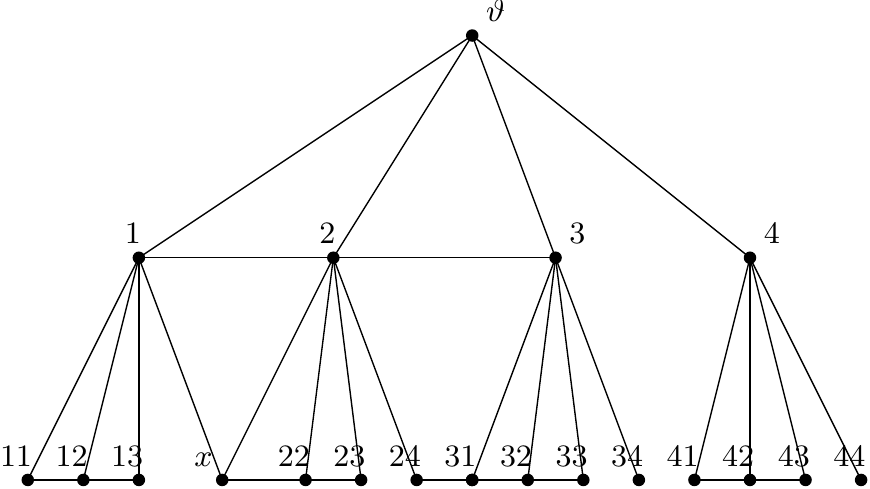}
    } \\
    \subfigure[]{
    \includegraphics[width=6.5cm]{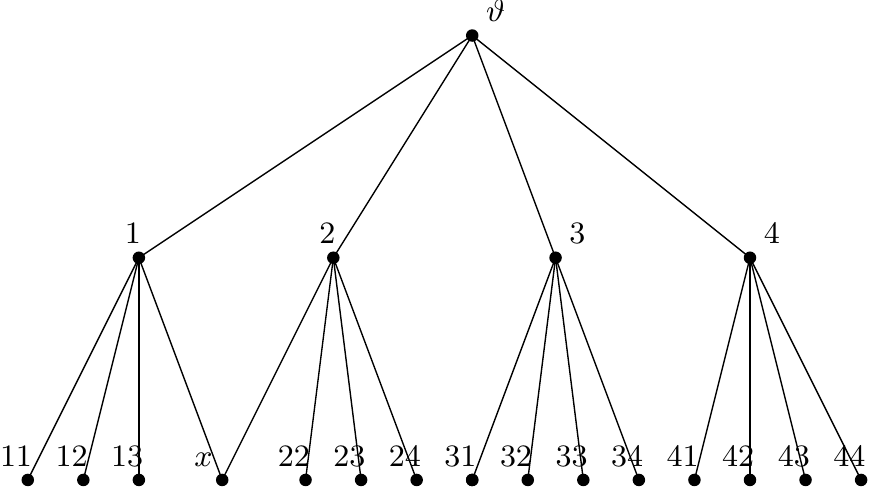}
    }\qquad
     \subfigure[]{
    \includegraphics[width=6.5cm]{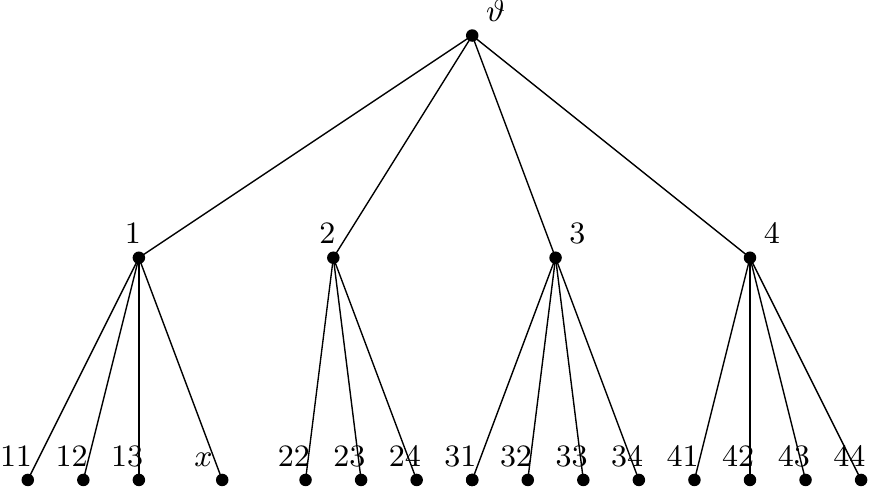}
    }
\caption{(a) the iterations of $[0,1]$; (b) the quotient space $(X^\sim, \E)$; (c) the reduced graph $(X^\sim, \E_v)$ by removing horizontal edges from $(X^\sim, \E)$; (d) the reduced tree $(X^\sim, \E_v^*)$ by modifying from $(X^\sim, \E_v)$.}\label{quotienttree}
\end{figure}
\end{Exa}

\begin{theorem}\label{th-quotient-tree}
Suppose the quotient space $(X^\sim,\E)$ is simple, and suppose the incidence matrix $A$ is $(B,\bu)$-rearrangeable (where $B$ and $\bu$ are defined as in Remark \ref{B-u-remark}). Then there exists a near-isometry between  $(X^\sim, {\mathcal E})$ and the reduced tree $(X^\sim, \E^*_v)$, so that $\partial (X^\sim,\E) \simeq \partial (X^\sim,\E^*_v)$.
\end{theorem}

\begin{proof}
Let $Y_1=(X^\sim, \E_v\cup \E_h), Y_2=(X^\sim, \E^*_v\cup \E_h)$ and $Y_3=(X^\sim, \E^*_v)$. Trivially, the identity map between $Y_1, Y_2$ is a near-isometry. The same proof of Theorem \ref{th-main} can be applied to build a near-isometry from $Y_2$ to $Y_3$. By composing the two near-isometries, we obtain the desired one.
\end{proof}

\section{Rearrangeable matrix and proof of Theorem \ref{th-main}}
In order to define a desired map from an augmented tree to the original tree, we need a combinatorial device to rearrange the vertices and edges of the graphs in a proper way. In a series of papers (\cite{LaLu13}\cite{DLL15}\cite{L13}), the author developed a concept of rearrangeable matrix to handle this case. In the present paper, we generalize the definition of rearrangeable matrix so that it helps to solve more general problems.

For a vector $\bu=[\bu_1,\dots, \bu_m]\in {\mathbb R}^{nm}$ with blocks $\bu_i=[u_{i1},\dots, u_{in}]\in {\mathbb R}^n$, we denote by $\bu^t=\left[\begin{array}{c}
          {\mathbf u}_1  \\
          \vdots \\
          {\mathbf u}_m
    \end{array} \right]$. We also let $\b 1=[1,\dots, 1]$ for convenience such that the involved matrix product is well-defined.

\begin{Def}\label{def rearrangeable}
Let $B$ be an $n\times n$ nonnegative integer matrix, let $\bu=[\bu_1,\dots, \bu_m]$ where $\bu_i=[u_{i1},\dots, u_{in}]\in {\mathbb Z}^n_+$. An $m\times m$ matrix $A$ is said to be  $(B,\bu)$-rearrangeable if for each row vector $\ba_i$ of $A$, there exists a nonnegative  matrix $C=[c_{ij}]_{p\times m}$ (rearranging matrix) where $p=\sum_{j=1}^nu_{ij}$ such that
\begin{equation*}
\ba_i=\b1 C, \quad  C\bu^t =
\left[\begin{array}{c}
          \sum_{j=1}^mc_{1j}{\mathbf u}_j  \\
          \vdots \\
          \sum_{j=1}^mc_{pj}{\mathbf u}_j
    \end{array} \right]:=\left[\begin{array}{c}
          {\mathbf c}_1  \\
          \vdots \\
          {\mathbf c}_p
    \end{array} \right]
\end{equation*} and $\#\{\bc_k: \bc_k=\bb_j, k=1,\dots, p\}=u_{ij}$, where $\bb_j$ is the $j$-th row vector of $B$.
\end{Def}

Note that when the matrix $B$ degenerates to an integer, the above definition coincides exactly with the original one in \cite{LaLu13} (or \cite{DLL15}).

\begin{Exa}\label{exa1}
Let $\bu=[\bu_1,\bu_2,\bu_3]$ where $\bu_1=[1,0], \bu_2=[2,1],\bu_3=[3,1]$. Let $$A=\left[\begin{array}{ccc}
          1 & 1 & 0\\
          1 & 2 & 1 \\
          1 & 2 & 2
    \end{array} \right],\quad B=\left[\begin{array}{cc}
          3 & 1\\
          2 & 1
    \end{array} \right].$$ Then $A$ is $(B,\bu)$-rearrangeable; $A^2$ is $(B^2,\bu)$-rearrangeable.
\end{Exa}

\begin{proof}
In fact, let $\ba_i$ be the $i$-th row vector of $A$ and $\bb_i$ the $i$-th row vector of $B$. For $\ba_1$, there exists a $1\times 3$ matrix $C=[1,1,0]$ such that $\ba_1= C$ and
$$C\bu^t=[1,1,0]\left[\begin{array}{c}
         \bu_1 \\
         \bu_2 \\
         \bu_3
    \end{array} \right]=\bu_1+\bu_2=\bb_1$$ which is the first row vector of $B$.

For $\ba_2$, there exists a $3\times 3$ matrix $C=\left[\begin{array}{ccc}
          1 & 1 & 0\\
          0 & 1 & 0 \\
          0 & 0 & 1
    \end{array} \right]$ such that $\ba_2= \b1 C$ and
$$C\bu^t=\left[\begin{array}{ccc}
          1 & 1 & 0\\
          0 & 1 & 0 \\
          0 & 0 & 1
    \end{array} \right]\left[\begin{array}{c}
         \bu_1 \\
         \bu_2 \\
         \bu_3
    \end{array} \right]=\left[\begin{array}{c}
         \bu_1+\bu_2 \\
         \bu_2 \\
         \bu_3
    \end{array} \right]=\left[\begin{array}{c}
         \bb_1 \\
         \bb_2 \\
         \bb_1
    \end{array} \right].$$

For $\ba_3$, there exists a $4\times 3$ matrix $C=\left[\begin{array}{ccc}
          1 & 1 & 0\\
          0 & 1 & 0 \\
          0 & 0 & 1 \\
          0 & 0 & 1
    \end{array} \right]$ such that $\ba_3= \b1 C$ and
$$C\bu^t=\left[\begin{array}{ccc}
          1 & 1 & 0\\
          0 & 1 & 0 \\
          0 & 0 & 1 \\
          0 & 0 & 1
    \end{array} \right]\left[\begin{array}{c}
         \bu_1 \\
         \bu_2 \\
         \bu_3
    \end{array} \right]=\left[\begin{array}{c}
         \bu_1+\bu_2 \\
         \bu_2 \\
         \bu_3 \\
         \bu_3
    \end{array} \right]=\left[\begin{array}{c}
         \bb_1 \\
         \bb_2 \\
         \bb_1 \\
         \bb_1
    \end{array} \right].$$

For $A^2=\left[\begin{array}{ccc}
          2 & 3 & 1\\
          4 & 7 & 4 \\
          5 & 9 & 6
    \end{array} \right], B^2=\left[\begin{array}{cc}
          11 & 4\\
          8 & 3
    \end{array} \right]$, we can use the same argument as above to show that $A^2$ is $(B^2,\bu)$-rearrangeable. The corresponding matrices $C$ for each row vector of $A^2$ are listed by $$\left[\begin{array}{ccc} 2 & 3 & 1 \end{array} \right], \quad \left[\begin{array}{ccc}
          3 & 4 & 0\\
          1 & 2 & 1 \\
          0 & 1 & 3
    \end{array} \right], \quad \left[\begin{array}{ccc}
         3 & 4 & 0\\
         0 & 1 & 3 \\
         0 & 1 & 3 \\
         2 & 3 & 0
    \end{array} \right].$$
\end{proof}

\begin{Exa}\label{exa-addition}
Let $\bu=[\bu_1, \dots, \bu_5]$ where $\bu_1=[1,0], \bu_2=[2,0], \bu_3=[0,2], \bu_4=[3,0], \bu_5=[1,2]$. Let $$A=\left[\begin{array}{rrrrr}
          1 &  1 & 1 & 0 & 0 \\
          1 &  1 & 2 & 1 & 0 \\
          1 &  0 & 1 & 0 & 1 \\
          1 &  1 & 3 & 2 & 0 \\
          1 &  1 & 1 & 0 & 2
    \end{array} \right],\quad B=\left[\begin{array}{cc}
          3 & 2\\
          1 & 2
    \end{array} \right].$$ Then $A$ is $(B,\bu)$-rearrangeable.
\end{Exa}

\begin{proof}
Analogous to Example \ref{exa1},  we can check the definition by letting the corresponding rearranging matrices $C$ for each row vector of $A$ as follows:
\begin{align*}\left[\begin{array}{ccccc} 1 & 1 & 1 & 0 & 0 \end{array}\right],
\left[\begin{array}{ccccc}
          1 & 1 & 1 & 0 & 0\\
          0 & 0 & 1 & 1 & 0
    \end{array} \right], \left[\begin{array}{ccccc}
          1 & 0 & 1 & 0 & 0\\
          0 & 0 & 0 & 0 & 1
    \end{array} \right],
\end{align*}
\begin{align*}
     \left[\begin{array}{ccccc}
          1 & 1 & 1 & 0 & 0\\
          0 & 0 & 1 & 1 & 0 \\
          0 & 0 & 1 & 1 & 0
    \end{array} \right],    \left[\begin{array}{ccccc}
          1 & 1 & 1 & 0 & 0\\
          0 & 0 & 0 & 0 & 1 \\
          0 & 0 & 0 & 0 & 1
    \end{array} \right].
\end{align*}
\end{proof}

\begin{Prop}\label{th-identity}
Let $A$ be $(B,\bu)$-rearrangeable as in Definition \ref{def rearrangeable}, then we have

(i) $A\bu^t=\bu^tB$;

(ii)  $A^k$ is $(B^k,\bu)$-rearrangeable for any $k>0$.
\end{Prop}

\begin{proof}
(i) Following the notation of Definition \ref{def rearrangeable}, for each row vector $\ba_i$ of $A$, then
$$\ba_i\bu^t=\b1 C \bu^t=\sum_{j=1}^p\bc_j=\sum_{j=1}^n u_{ij}\bb_j=\bu_i B.$$

We prove (ii) by induction. Suppose $A^k$ is $(B^k,\bu)$-rearrangeable for some $k\ge 1$. For $k+1$, we let ${\boldsymbol \alpha}_i$ be the $i$-th row vector of $A^k$ and ${\boldsymbol \beta}_j$ the $j$-th row vector of $B^k$. Then there exists a $p\times m$ matrix $C$ where $p=\sum_{j=1}^n u_{ij}$ such that  ${\boldsymbol \alpha}_i=\b1 C$ and $$C\bu^t=\big[\underbrace{{\boldsymbol \beta}_1,\dots,{\boldsymbol \beta}_1}_{u_{i1}},\dots, \underbrace{{\boldsymbol \beta}_n,\dots,{\boldsymbol \beta}_n}_{u_{in}}\big]^t,
$$
where the transpose means transposing the row of matrices into a column of matrices (without transposing the ${\boldsymbol\beta}_j$ itself).

Since  $$A^{k+1}=A^k A=\big[{\boldsymbol \alpha}_1 A, \dots, {\boldsymbol \alpha}_m A\big]^t\quad\text{and}\quad B^{k+1}=B^k B=\big[{\boldsymbol \beta}_1 B, \dots, {\boldsymbol \beta}_n B\big]^t.$$ It follows that ${\boldsymbol \alpha}_iA=\b1 CA$. Letting $\tilde{C}=CA$, then by (i), we have
$$\tilde{C}\bu^t=CA\bu^t=C\bu^t B=\big[\underbrace{{\boldsymbol \beta}_1B,\dots,{\boldsymbol \beta}_1B}_{u_{i1}},\dots, \underbrace{{\boldsymbol \beta}_nB,\dots,{\boldsymbol \beta}_nB}_{u_{in}}\big]^t.$$ Therefore, $A^{k+1}$ is $(B^{k+1},\bu)$-rearrangeable by Definition \ref{def rearrangeable}.
\end{proof}

\noindent {\bf Proof of Theorem  \ref{th-main}:}
By the assumption, let $\{[T_i]\}_{i=1}^m$ be the equivalence classes of horizontal components in $(X,\E)$, and let $\{[\bt_j]\}_{j=1}^n$ be the equivalence classes of vertices in $(X, \E_v)$. Without loss of generality, we can assume that
\begin{equation}\label{eq.WLOG}
\max_{1\le i\le m}\#T_i\le \min_{1\le i\le n}\sum_{j=1}^n b_{ij}.
\end{equation}
For otherwise, by the definition of incidence matrix $B$ of $(X,\E_v)$, each row sum and each column sum of $B$ must be positive. Hence every row sum of $B^k$ will go to infinity as $k$ does. Denote by $b_{ij}^k$ the entries of $B^k$, and let $k$ be sufficiently large such that
$$\max_{1\le i\le m}\#T_i\le \min_{1\le i\le n}\sum_{j=1}^n b^k_{ij}.$$
By Proposition \ref{th-identity}, $A^k$ is $(B^k,\bu)$-rearrangeable. The IFS of the $k$-th iteration of $\{S_i\}_{i=1}^m$ has symbolic space $X'=\bigcup_{n=0}^\infty X_{kn}$ and the augmented tree has incidence matrix $A^k$; moreover, the two hyperbolic boundaries $\partial X'$ and $\partial X$ are identical. Hence we can consider $A^k$ instead if \eqref{eq.WLOG} is not satisfied.

Let $X=(X,{\mathcal E}), Y=(X,{\mathcal E}_v)$.  We define $\sigma: X\to Y$ to be a one-to-one map  from $X_n$ (in $X$) to $X_n$ (in $Y$) inductively as follows: Let
$$
\sigma(\vartheta)=\vartheta \quad \hbox {and} \quad \sigma(x)=x, \quad x\in X_1.
$$
Suppose  $\sigma$ is defined on $X_n$ such that

(1) for every horizontal component $T$, $\sigma (T)$ have the same parent (using \eqref{eq.WLOG}), i.e.,
\begin{equation*}
\sigma (x)^{-1} = \sigma (y)^{-1} \qquad  \forall \ x, y \in T;
\end{equation*}

(2) $x\sim\sigma(x)$, i.e., $x$ and  $\sigma(x)$ lie in the same class $[\bt_j]$ for some $j$.

In order to define the map $\sigma$ on  $X_{n+1}$, we note that  $T$ in $X_n$ gives rise  to horizontal components in $X_{n+1}$, which are accounted by the incidence matrix $A$. We write the descendants of $T$ as $\bigcup_{k=1}^\ell Z_k$ where $Z_k$ are horizontal components consisting of $X_{n+1}$-descendants of $T$. If $T$ belongs to some equivalence class  $[T_i]$ (corresponding to $\bu_i=[u_{i1},\dots, u_{in}]$),  then $T$ consists of $u_{ij}$ vertices belonging to the class $[\bt_j]$ where $j=1,\dots, n$. The total number of vertices of $T$ is $\# T=\sum_{j=1}^nu_{ij}:=p$. By the definition of incidence matrix $A$, the total number of the $X_{n+1}$-descendants of $T$ is $$\sum_{k=1}^\ell \#Z_k=(\ba_i\bu^t)\cdot\b1=\sum_{j=1}^m\sum_{k=1}^na_{ij}u_{jk}.$$

Since  $A$ is $(B,\bu)$-rearrangeable, there exists a nonnegative  matrix $C=[c_{ij}]_{p\times m}$ such that
\begin{equation*}
\ba_i=\b1 C, \quad  C\bu^t =
\left[\begin{array}{c}
          \sum_{j=1}^mc_{1j}{\mathbf u}_j  \\
          \vdots \\
          \sum_{j=1}^mc_{pj}{\mathbf u}_j
    \end{array} \right]:=\left[\begin{array}{c}
          {\mathbf c}_1  \\
          \vdots \\
          {\mathbf c}_p
    \end{array} \right].
\end{equation*}
According to the matrix $C$, we decompose $\ba_i$ into $p$ groups as follows: for each $1\le s\le p$ and $1\le j\le m$, we choose $c_{sj}$ of those $Z_k$ that belong to $[T_j]$, and denote by $\Lambda_s$ the set of all the chosen $k$ with $1\le j\le m$. Then the index set $\{1,2,\dots, \ell\}$ can be written as a disjoint union
$$\{1,2, \dots, \ell\}=\bigcup_{s=1}^p\Lambda_s.$$
Hence  $\bigcup_{k=1}^\ell Z_k$ can be rearranged as $p$ groups $\Lambda_s$, and each $\Lambda_s$ corresponds to $\bc_s$. Namely,
\begin{equation*}
\bigcup_{k=1}^{\ell}Z_k=\bigcup_{k\in \Lambda_1}Z_{k}\ \cup \
\cdots \cup \ \bigcup_{k\in \Lambda_{p}}Z_{k}.
\end{equation*}

\begin{figure}[h]
  \centering
  \includegraphics[width=14cm]{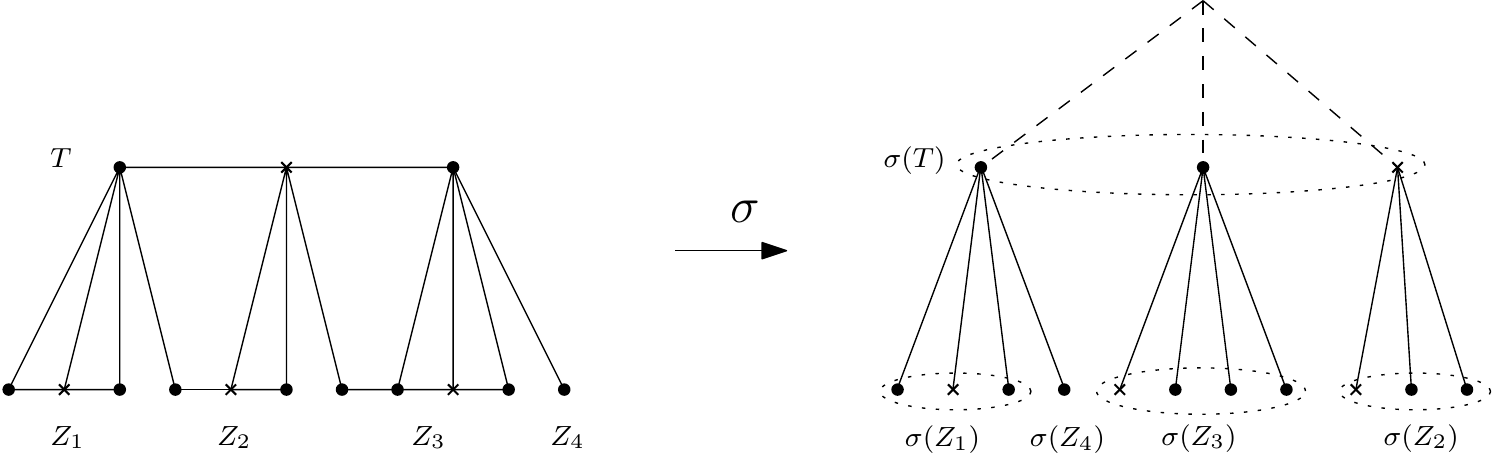}
\caption{An illustration of a rearrangement by $\sigma$ with $\ell=4, p=3$, the {\tiny $\bullet , \ \times$} denote two classes of vertices.}\label{fig.re}
\end{figure}

Suppose the component  $T = \{x_1,  \dots, x_p\} (\subset X_n)$, we have defined $\sigma$ on $X_n$ and $\sigma (T) = \{y_1 = \sigma (x_1), \ \dots ,\ y_p = \sigma (x_p)\}$ such that $x_i, y_i$ are equivalent. As $T\in[T_i]$, by the previous argument, among $x_i's$  (or $y_i's$), there are $u_{ij}$ vertices belonging to the class $[\bt_j]$, $j=1,\dots, n$.  Since $\#\{\bc_s: \bc_s=\bb_j, s=1,\dots, p\}=u_{ij}$, for $\bc_s=\bb_j$, we have
$$\sum_{k\in \Lambda_s}\#Z_{k}=\bb_j\cdot\b1=\sum_{k=1}^nb_{jk}$$  which is equal to the total number of descendants of a vertex lying in $[\bt_j]$ (due to the incidence matrix $B$). We can define a one-to-one map $\sigma$ on $\bigcup_{k=1}^{\ell}Z_k$ by assigning ${\bigcup}_{k\in\Lambda_s}Z_{k}$ to  the descendants of $y_i$ which belongs to $[\bt_j]$. Moreover, $\bc_s=\bb_j$ also means that the number of vertices belonging to $[\bt_k]$ in  ${\bigcup}_{k\in\Lambda_s}Z_{k}$ is the same as the number of descendants of the $y_i$ which belong to $[\bt_k]$ for $k=1,\dots, n$. Hence we can require the $x$ and its image $\sigma(x)$ are equivalent (see Figure \ref{fig.re}).  Therefore, $\sigma$ is well-defined on the descendants of $T$ and satisfies conditions (1)(2). By applying the same construction of $\sigma$ on the descendants of every horizontal component in $X_n$. It follows that $\sigma$ is well-defined on $X_{n+1}$ as desired. Inductively, $\sigma$ can be defined from $X$ to $Y$ bijectively.

The proof that $\sigma$ is a near-isometry is the same as in \cite{LaLu13}. Let $x,y\in X$, $\pi(x,y)$  be the canonical geodesic connecting them, which can be written as
$$\pi(x,y)=[x,x^{-1},\dots, x^{-n},z_1,\dots, z_k, y^{-m},\dots, y^{-1}, y]
$$
where $[x^{-n},z_1,\dots, z_k,y^{-m}]$ is the horizontal part and $[x,x^{-1},\dots, x^{-n}], [y^{-m},\dots, y^{-1}, y]$ are vertical parts. Clearly, $[x^{-n},z_1,\dots, z_k,y^{-m}]$ must be included in a horizontal component of $X$, we denote it by $T$.  By \eqref{eq.cano.geo.identity}, it follows  that
$$
d(x,y)=|x|+|y|-2l+h, \quad d(\sigma(x),\sigma(y))=|\sigma(x)|+|\sigma(y)|-2l'+h',
$$ where $l, l'$ and $h, h'$ are levels and lengths of $\pi(x, y), \pi(\sigma(x), \sigma(y))$, respectively. Then we have $$|d(\sigma(x),\sigma(y))-d(x,y)|\le |h-h'|+2|l-l'|\le c+2|l-l'|$$
where $c$ is a hyperbolic constant as in Theorem \ref{th2.3}. If $T$ is a singleton, then $l'=l$; otherwise, the elements of $\sigma(T)$ share the same parent. Hence $l'=l-1$, and
$$
|d(\sigma(x),\sigma(y))-d(x,y)|\leq c+2.
$$
That completes the proof of the theorem.

\section{Lipschitz equivalence of self-similar sets}

Lipschitz equivalence of self-similar sets is an interesting problem, which has been undergoing considerable development recently. It is well-known that the Hausdorff dimension is invariant under a bi-Lipschitz map. However, it is impractical to find the Hausdorff dimension of a fractal by using the relation of Lipschitz equivalence to other ones. Usually, Hausdorff dimension is a prerequisite for Lipschitz equivalence in the studies. We know that, under the OSC, the Hausdorff dimension of a self-similar set can be calculated easily (\cite{Hu81}\cite{F}). In the absence of  the OSC, it is much harder to compute the Hausdorff dimension. Nevertheless, the following formula  is very useful for a large class of self-similar sets with overlaps.

\begin{Lem}[\cite{NgWa01}]\label{thm-finitetype}
Let $K$ be the self-similar set on ${\mathbb R}^d$ defined by (\ref{set equ}). If the finite type condition holds,  then the Hausdorff dimension of $K$ is given by $$\dim_H K=\frac{\ln \rho}{-\ln r}$$ where $\rho$ is the spectral radius of a matrix associated with the finite type condition, $r$ is the minimum contraction ratio.
\end{Lem}

The following result was  recently proved in \cite{LaWa16}, which establishes a very important relationship between the hyperbolic boundary of an augmented tree and the self-similar set.

\begin{theorem}[\cite{LaWa16}]\label{th-holder-equ}
With the same notation as in Section 2, the augmented tree $(X, {\mathcal E})$ is hyperbolic, and the hyperbolic boundary $\partial X$ of $(X, {\mathcal E})$ is H\"older equivalent to the self-similar set $K$, i.e., there exists a natural  bijection $\Phi: \partial X \to K$ and a constant $C>0$ such that
\begin{equation}\label{eq-holder}
C^{-1}|\Phi(\xi)-\Phi(\eta)|\le \rho_a^\alpha(\xi, \eta)\le C|\Phi(\xi)-\Phi(\eta)|, \quad \forall \ \xi, \eta\in \partial X,
\end{equation}  with $\alpha=-(\log r)/a$.
\end{theorem}

By using the above theorem, we can obtain the Lipschitz equivalence of self-similar sets.

\begin{theorem} \label{th-lip-sets}
Let $K$ and $K'$ be self-similar sets that are generated by two IFS's as in (\ref{set equ}) with the same ratios $\{r_i\}_{i=1}^N$. If the corresponding augmented trees $X, Y$ are both simple and their incidence matrices are $(B, \bu)$-rearrangeable and $(B, \bu')$-rearrangeable respectively, where $B$ is the common incidence matrix of their tree structures.  Then $K$ and $K'$ are Lipschitz equivalent, and are also Lipschitz equivalent to a dust-like self-similar set with the same ratios $\{r_i\}_{i=1}^N$.

If we replace the above $X, Y$ by their quotient spaces $X^\sim, Y^\sim$,  then $K$ and $K'$ are still Lipschitz equivalent, and are Lipschitz equivalent to a Cantor-type set (may be not a dust-like self-similar set).
\end{theorem}

\begin{proof}
It follows from Theorem \ref{th-main} and Proposition \ref{th.simple tree} that $$\partial (X, \E) \simeq \partial (X, \E_v)=\partial (Y, \E_v')\simeq \partial (Y, \E').$$

Let $\varphi: \partial X  \to \partial Y$ be a bi-Lipschitz map. By Theorem \ref{th-holder-equ}, there exist two bijections $\Phi_1:\partial X\to K$ and $\Phi_2:\partial Y\to K'$ satisfying \eqref{eq-holder} with constants $C_1,C_2$, respectively.  Now we define $\tau: K \to K'$ as
$$
\tau = \Phi_2\circ \varphi\circ \Phi_1^{-1}.
$$
Then
$$
\begin {aligned}
|\tau(x) -\tau (y)| & \leq
C_2\ \rho_a(\varphi\circ\Phi_1^{-1}(x),\varphi\circ\Phi_1^{-1}(y))^\alpha\\
&\leq C_2C_0^\alpha\ \rho_a(\Phi_1^{-1}(x),\Phi_1^{-1}(y))^\alpha\\
&\leq
C_2C_0^\alpha C_1\ |x-y|.
\end{aligned}
$$
Let $C' = C_2C_0^\alpha C_1$, then $ |\tau (x) - \tau (y)| \leq C'|x -y|$. Moreover, we have ${C'}^{-1} |x-y| \leq |\tau (x) - \tau (y)|$ by using another inequality of \eqref{eq-holder}. Therefore $\tau : K \to K'$ is a bi-Lipschitz map and $K\simeq K'$.

If we regard $(X, \E_v)$ as the augmented tree of an IFS that is strongly separated, then apply the above conclusion to obtain that $K$ or $K'$ is Lipschitz equivalent to the dust-like self-similar set.

The second part follows from Theorem \ref{th-quotient-tree} and Proposition \ref{th.simple tree} by the same argument. Since  $(X^\sim ,\E^*_v)$  is a tree, its hyperbolic boundary is a Cantor-type set. However, due to the identification of vertices, it seems hard to regard $(X^\sim ,\E^*_v)$   as an augmented tree of an IFS that is strongly separated.
\end{proof}

Finally we illustrate our main results by two concrete examples. The first one with the OSC was considered in \cite{XiRu07} (or \cite{RuWaXi12}\cite{XiXi13}) by using algebraic properties of contraction ratios,  we will give a different proof from our perspective; the second one without the OSC is completely new. The higher dimensional case can also be handled but need more complicated discussions, we omit it here.

\begin{Exa}
Let $S_1(x)=rx, S_2(x)=r^2x+1-2r^2, S_3(x)=r^2x+1-r^2$ and $S'_2(x)=r^2x+\frac{1+r-2r^2}{2}$ where $0<r<1$. Let $K$ (resp. $K'$) be the self-similar set generated by the IFS $\{S_1, S_2, S_3\}$ (resp. $\{S_1, S'_2, S_3\}$) (see Figure \ref{fig-dust-touch}). Then $K\simeq K'$.
\end{Exa}

\begin{proof}
Obviously both the two IFS's satisfy the OSC, and $K'$ is dust-like while $K$ is not as $S_2(K)\cap S_3(K)=\{1-r^2\}$ (see Figure \ref{fig-dust-touch}). Let $(X, \E)$ be the augmented tree of $K$  defined as in Section 2. Then $(X, \E_v)$  can be regarded as the augmented tree of $K'$. By a few calculations, it can be seen that there are only  five equivalence classes of the horizontal components in $(X, \E)$: $[T_1]=[\{\vartheta\}]$; $[T_2]=[\{2, 3\}]$; $[T_3]=[\{12, 13\}]$; $[T_4]=[\{22, 23, 311\}]$; $[T_5]=[\{122, 123, 131\}]$. Hence $(X,\E)$ is simple and incidence matrix is
$$A=\left[\begin{array}{rrrrr}
          1 &  1 & 1 & 0 & 0 \\
          1 &  1 & 2 & 1 & 0 \\
          1 &  0 & 1 & 0 & 1 \\
          1 &  1 & 3 & 2 & 0 \\
          1 &  1 & 1 & 0 & 2
    \end{array} \right].$$
For the tree $(X, \E_v)$, the equivalence classes of vertices are $[\bt_1]=[\{\vartheta\}]$, $[\bt_2]=[\{13\}]$ and the incidence matrix is
$$B=\left[\begin{array}{rr}
          3 &  2  \\
          1 & 2
    \end{array}\right].
$$
\begin{figure} [h]
    \centering
    \subfigure[]{
      \includegraphics[width=6.5cm]{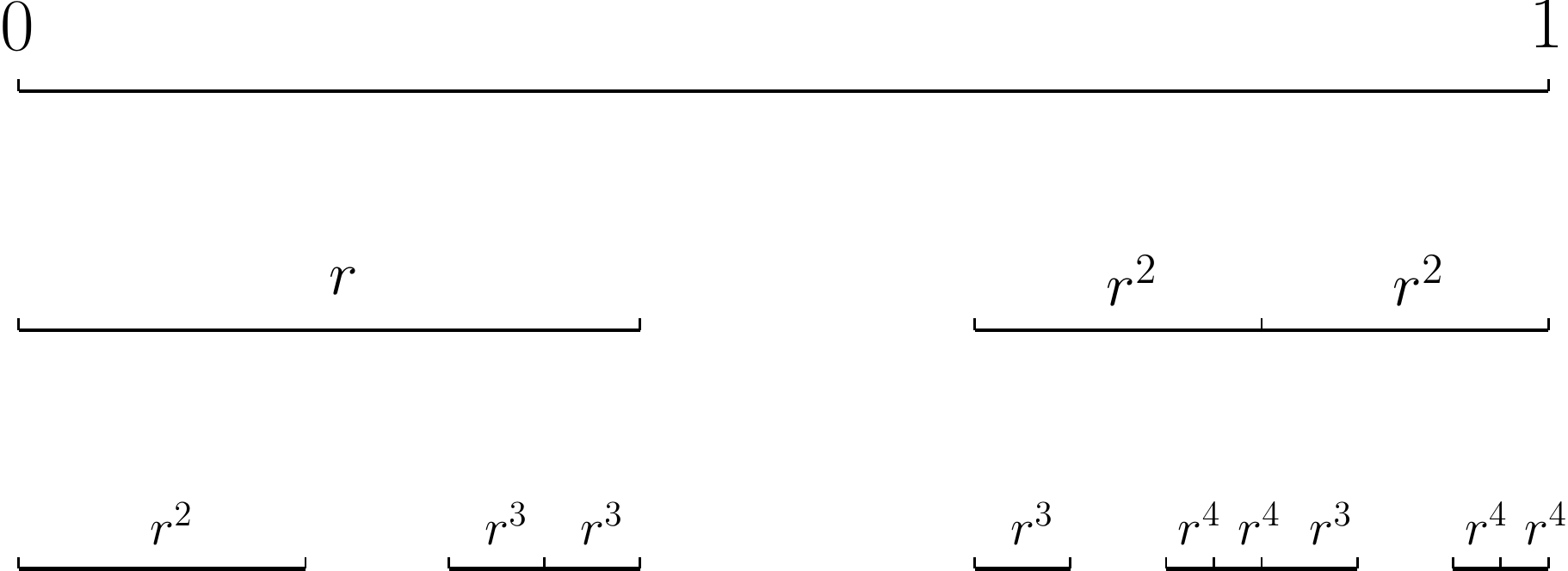}
     }\qquad
      \subfigure[]{
      \includegraphics[width=6.5cm]{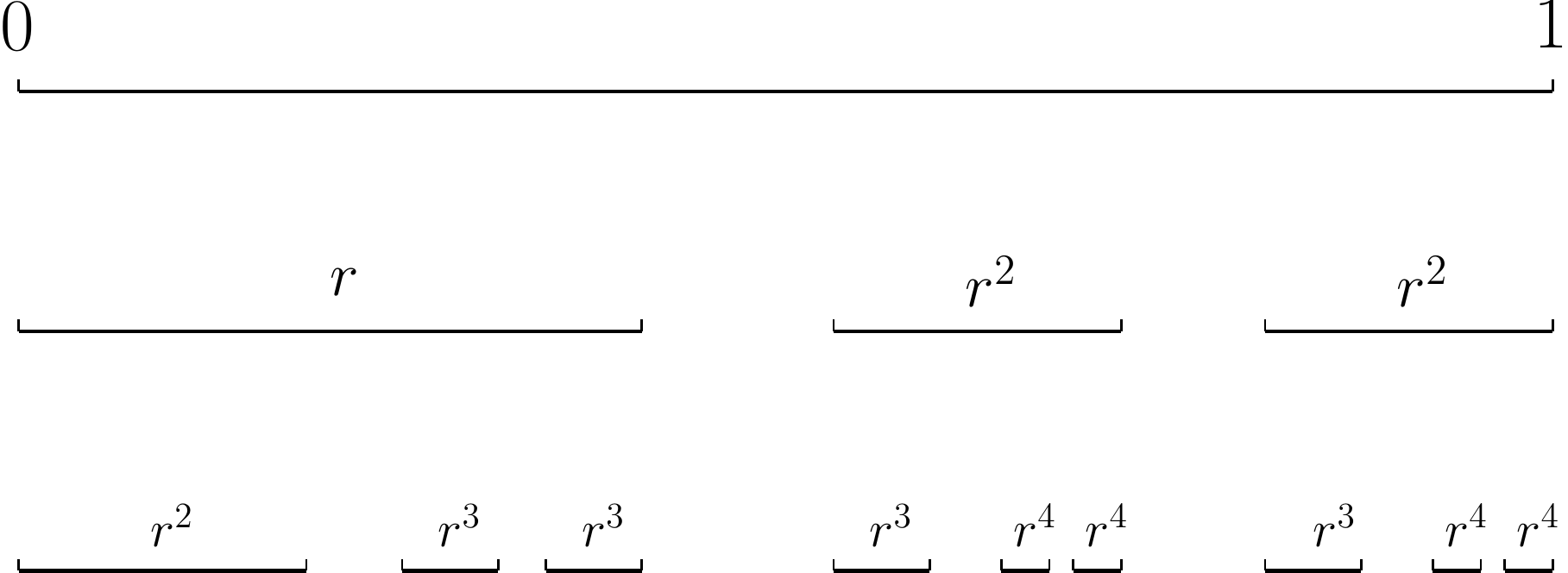}
     }
\caption{(a) the iterations of $K$; (b) the iterations of $K'$.}\label{fig-dust-touch}
\end{figure}

Let $\bu=[\bu_1, \dots, \bu_5]$ be the representation of $[T_1],\dots, [T_5]$ with respect to $[\bt_1], [\bt_2]$, where $\bu_1=[1,0], \bu_2=[2,0], \bu_3=[0,2], \bu_4=[3,0], \bu_5=[1,2]$. It follows from Example \ref{exa-addition} that the matrix $A$ is $(B, \bu)$-rearrangeable. Therefore $K\simeq K'$ by Theorem \ref{th-lip-sets}
\end{proof}

\begin{Exa}
Let $S_1(x)=\frac{1}{4}x, S_2(x)=\frac{1}{4}(x+\frac34), S_3(x)=\frac{1}{4}(x+\lambda), S_4(x)=\frac{1}{4}(x+3)$ be an IFS where $\lambda\in [0, 3]$, and let $K_\lambda$ be the self-similar set. Then $K_{\lambda}\simeq K_2$ if and only if $\lambda\in [\frac74, 2]$.
\end{Exa}

\begin{proof}
By the assumption, $K_\lambda\subset [0,1]$.   It is easy to see that $S_{14}=S_{21}$ (see Figure \ref{fig-exa}), hence the OSC does not hold for all $\lambda$. However, if $\lambda\in [\frac 74, 2]$, the IFS satisfies the finite type condition \cite{NgWa01} (hence satisfies the WSC \cite{DeLaNg13}). By Lemma \ref{thm-finitetype},  all such self-similar sets $K_\lambda$ have the common Hausdorff dimension
$$\dim_H{K_\lambda}=\frac{\log (2+\sqrt{3})}{\log 4}$$ where $2+\sqrt{3}$ is the spectral radius of the  following matrix associated with the finite type condition
$$\left[\begin{array}{rrr}
          2 &  1 & 1 \\
          1 & 2 & 1 \\
          0 & 2 & 1
    \end{array} \right].$$

\begin{figure} [h]
    \centering
    \subfigure[]{
      \includegraphics[width=6.5cm]{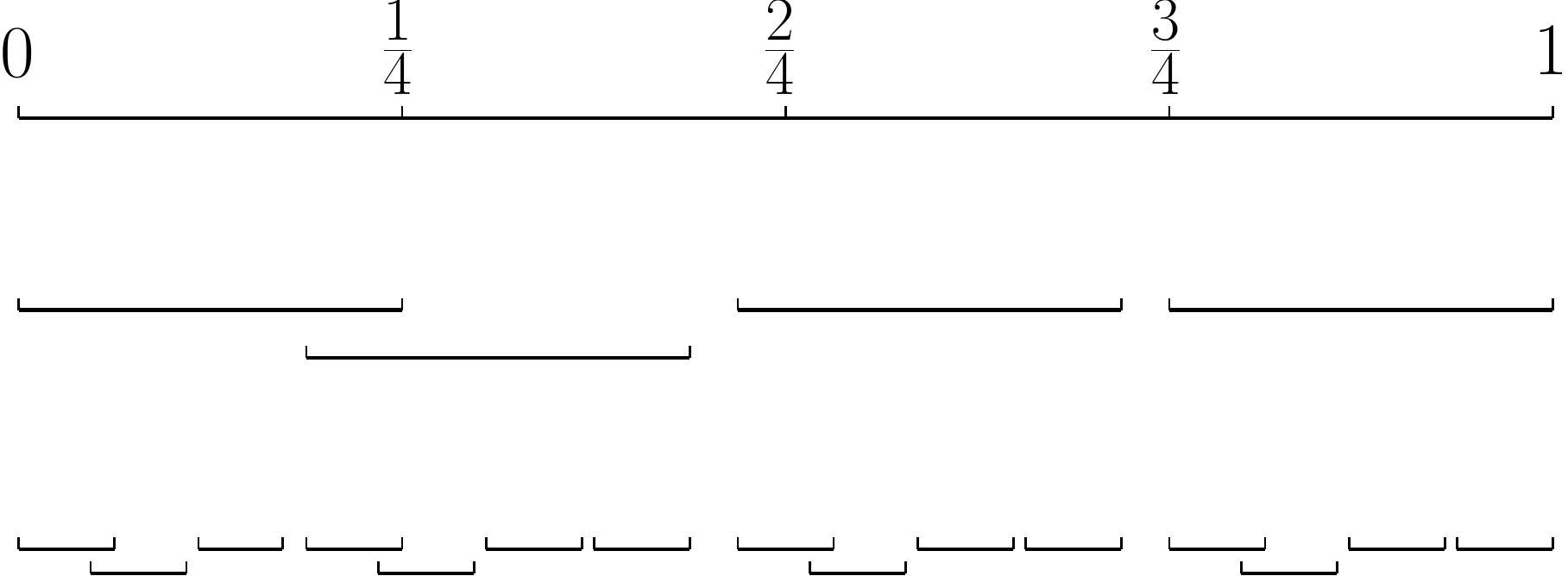}
     }\qquad
      \subfigure[]{
      \includegraphics[width=6.5cm]{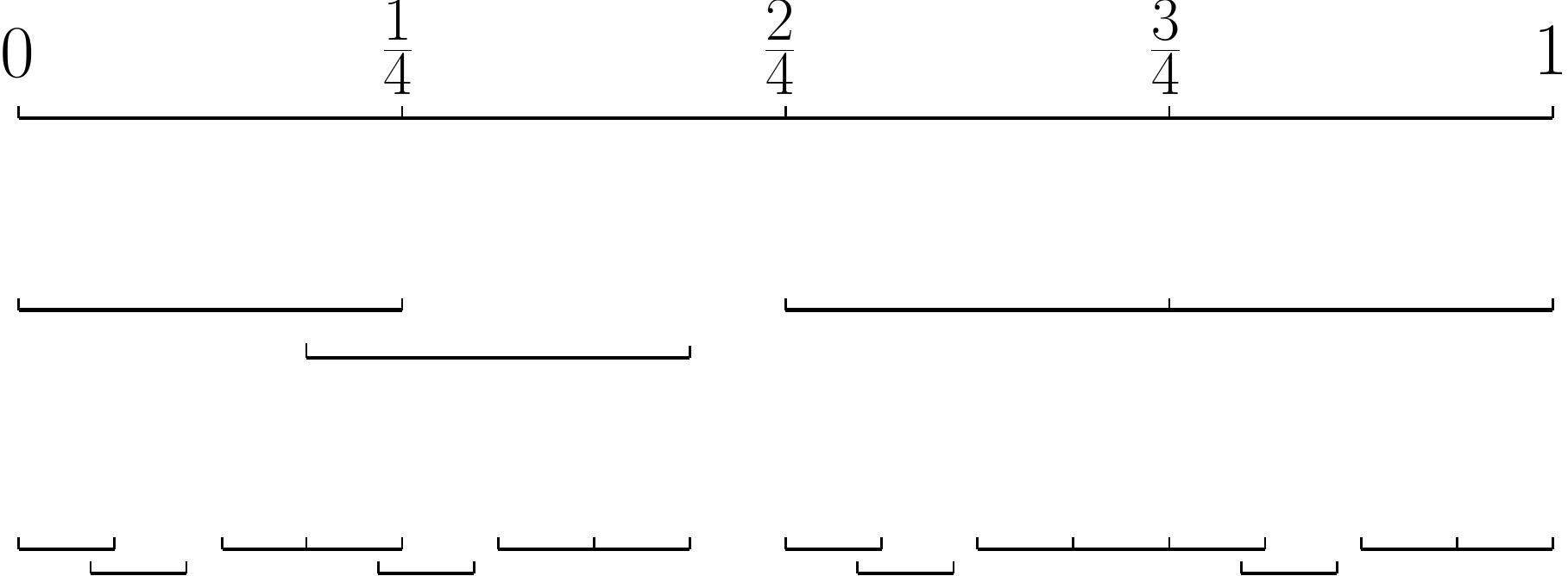}
     }
\caption{(a) the iterations of $K_\lambda$ where $\lambda\in (\frac74, 2)$; (b) the iterations of $K_2$.}\label{fig-exa}
\end{figure}

If $\lambda\in (\frac74, 2)$, then  all self-similar sets $K_\lambda$ are Lipschitz equivalent to each other as their augmented trees have the same quotient space (see Figure \ref{fig-exa} (a)). Hence it suffices to show $K_{\frac74}\simeq K_{\lambda_0}\simeq K_2$ where $\lambda_0=\frac{15}{8}$. By Example \ref{example-quotienttree}, the quotient space $(X^\sim, \E)$ of $K_{\frac74}$ is simple. Similarly  it can be seen that the quotient spaces of $K_{\lambda_0}$ and  $K_2$ are also simple. We write the incidence matrices of $K_{\frac{7}{4}}, K_{\lambda_0}, K_2$ as follows
\begin{align*}
A_1=\left[\begin{array}{rrr}
          1 &  1 & 0 \\
          1 & 2 & 1 \\
          1 & 2 & 2
    \end{array} \right]; \quad
 A_2=\left[\begin{array}{rr}
          2 & 1\\
          3 & 2
    \end{array}\right]; \quad
 A_3=\left[\begin{array}{rrrr}
          1 & 1 &  0 & 1  \\
          1 & 1 &  1 & 0 \\
          2 & 2 & 1 &  2  \\
          1 & 1 & 1 & 2
    \end{array}\right].
\end{align*}
Moreover, the corresponding reduced trees $(X^\sim, \E^*_v)$  share the same incidence matrix
$B=\left[\begin{array}{rr}
          3 &  1  \\
          2 & 1
    \end{array}\right]$, thus they are graph isomorphic by Proposition  \ref{th.simple tree}.

Let $\bu=[\bu_1,\bu_2,\bu_3]$ be the representation for the quotient space of $K_{\frac74}$ where $\bu_1=[1,0], \bu_2=[2,1],\bu_3=[3,1]$. Then $A_1$ is $(B, \bu)$-rearrangeable by Example \ref{exa1}. By checking the definition, similarly we have $A_2$ is $(B, \bu)$-rearrangeable where $\bu=[\bu_1,\bu_2]$ is a representation for the quotient space of $K_{\lambda_0}$ where $\bu_1=[1,0], \bu_2=[1,1]$; and $A_3$ is $(B, \bu)$-rearrangeable where $\bu=[\bu_1,\bu_2,\bu_3,\bu_4]$  is a representation for the quotient space of $K_2$ where $\bu_1=[2,0], \bu_2=[1,1], \bu_3=[2,1], \bu_4=[3,1]$. From Theorem \ref{th-quotient-tree}, it concludes   that $\partial (X^\sim,\E) \simeq \partial (X^\sim,\E^*_v)$. Hence the corresponding self-similar sets are Lipschitz equivalent by Theorem \ref{th-lip-sets}.

If $\lambda\in [0,\frac74)\cup (2, 3]$, then the overlaps of the IFS become more complicated but it is not hard to see that $\dim_H{K_\lambda}< \frac{\log (2+\sqrt{3})}{\log 4}$ \cite{NgWa01}. Therefore, the $K_{\lambda}$ can not be Lipschitz equivalent to $K_2$.
\end{proof}

\section{Topology of augmented trees}
In this section, we provide a criterion for the hyperbolic boundary $\partial X$ to be totally disconnected, which answers a question posed in \cite{LaLu13}.

\begin{Lem} \label{lem-tree1}
Let $(X,{\mathcal E})$ be an augmented tree defined in Section 2, and let $\pi(x,y)$ be a geodesic with $x,y\in X_n$ and be of the trapezoidal form (not necessarily canonical).

(i) if $\pi(x,y)\nsubseteq X_n$, then $\pi(x^{-1}, y^{-1})$ is a geodesic;

(ii) if $\pi(x,y)\subseteq X_n$, then $x^{-1}, y^{-1}$ can be joined by a path from $X_{n-1}$ of length less than $c$ which is the hyperbolic constant as in Theorem \ref{th2.3}.
\end{Lem}

\begin{proof}
(i) is obvious; (ii) The assumption implies that $\pi(x,y)$ is a horizontal geodesic. Hence $d(x,y)\le c$ by Theorems \ref{th-holder-equ} and \ref{th2.3}. By taking the path formed by the parents of the vertices in $\pi(x,y)$, we see that $x^{-1}, y^{-1}$ are joined by a path in $X_{n-1}$ with length less than $c$.
\end{proof}

\begin{Lem}
Let $(X,{\mathcal E})$ be an augmented tree of bounded degree. Let $\{T_n\}_{n\ge 1}$ be a sequence of horizontal components such that $\lim_{n\to\infty}\#T_n=\infty$. Then there exist two sequences of points $\{a_n\}_{n\ge 1}, \{b_n\}_{n\ge 1}$ with $a_n,b_n\in T_n$ and $\lim_{n\to\infty}d(a_n,b_n)=\infty$.
\end{Lem}

\begin{proof}
Without loss of generality, we assume $T_n\subseteq X_n$. Suppose the statement is false, then there exists $\eta>0$ such that $$d(a,b)\le \eta\quad \text{for any}\quad a,b\in T_n, n\ge 1.$$
This implies all geodesics connecting $a,b\in T_n$ have length less than $\eta$. There exists $\ell>0$ (independent of $n$) such that all such geodesics are contained in $\cup_{j=n-\ell}^n X_j$ (they are trapezoidal, but not necessary canonical). Consider $T_n^{(-k)}$, the set of $k$-th generation ancestors of $T_n$. It is clear that $T_n^{(-k)}$ is connected by the definition of horizontal edge set $\E_h$. Note that if $a,b\in X_n$ and $\pi(a,b)\subseteq X_n$ (case (ii) in Lemma \ref{lem-tree1}). Then $a^{-1}, b^{-1}$ are joined by a path in $X_{n-1}$ of length less than $c$, and inductively, $a^{-\ell}, b^{-\ell}$ are joined by a path in $X_{n-\ell}$ of length less than $c$. If $a,b\in X_n$ but $\pi(a,b)\nsubseteq X_n$ (case (i) in Lemma \ref{lem-tree1}), then there exists $k\le \ell$ such that $\pi(a^{-k}, b^{-k})\subseteq X_{n-k}$. In this case we can derive as the above that $a^{-\ell}, b^{-\ell}$  again can be joined by a path in $X_{n-\ell}$ of length less than $c$.

From the above discussion, we see that for any two vertices in $T_n^{(-\ell)}(\subseteq X_{n-\ell})$ can be joined by a path in $X_{n-\ell}$ of length less than $c$. As $(X, {\mathcal E}_h)$ has bounded degree, say $d_1$. Then $$\#T_n^{(-\ell)}\le d_1^c.$$
On the other hand, $(X,{\mathcal E}_v)$ has bounded degree, say $d_2$. Then $$\#T_n^{(-\ell)}\ge \frac {\#T_n}{d_2^\ell}.$$
This leads to contradiction as $\lim_{n\to\infty}\#T_n=\infty$.
\end{proof}

\begin{Cor}
Under the same hypothesis of the above lemma, the set of limit points of $\{T_n\}_{n\ge 1}$ is not a singleton.
\end{Cor}

\begin{Lem}\label{lem-tree3}
Let $(X,{\mathcal E})$ be an augmented tree, let $p(u_0,u_1,\dots, u_k)\subset X_n$ ($n$-th level) be a non-self-intersecting path. If $d(u_0, d_k)\ge 3$, then there exists $1\le m_0\le k$ such that $\frac13 d(u_0, u_k)\le d(u_0, u_{m_0})\le \frac23 d(u_0,u_k)$.
\end{Lem}

\begin{proof}
Let $d_m=d(u_0, u_m), 1\le m\le k$, then either $0\le d_{m+1}-d_m\le 1$ or $0\le d_m-d_{m+1}\le 1$. That is, $d_m$, as a function of $m$, increases either by at most $1$ or decreases by at most $1$ at each step. Then the lemma follows.
\end{proof}

\begin{theorem}\label{th-bdy is total.dis}
Let $(X, {\mathcal E})$ be an augmented tree of bounded degree. Then the hyperbolic boundary $\partial X$ (or the self-similar set $K$) is totally disconnected if and only if the sizes of horizontal components are uniformly bounded.

The theorem remains valid if we replace $X$ by its quotient space $X^\sim$.
\end{theorem}

\begin{proof}
We prove the theorem by contrapositive.  The sufficiency is essentially the same as in \cite{LaLu13} by using topological arguments, we sketch the main idea here. If $\partial X$  is not totally disconnected, so is the self-similar set $K$ by Theorem \ref{th-holder-equ}. Then there is a connected component $C\subset K$ contains more that one point. Note that for any $n>0$, $K = \bigcup_{x\in{X_n}} K_x$ where $K_x:=S_x(K)$. Let $K_{x_1} \cap C \not = \emptyset$. If $C\setminus K_{x_1} \not = \emptyset$,  then it is a relatively open set in $C$, and as $C$ is connected,
$$
\partial_C (C \setminus K_{x_1}) \cap   \partial_C (K_{x_1}\cap C) \not = \emptyset.
$$
($\partial_C (E)$ means the relative boundary  of $E$ in $C$). Let $w$ be in the intersection,  there exists $x_2 \in {X_n}$ such that $ w \in K_{x_1} \cap  K_{x_2}$ and $K_{x_2}\cap (C \setminus K_{x_1}) \not = \emptyset$.

Inductively, if $\bigcup _{j=1}^kK_{x_j}$ does not cover $C$,  then we can repeat the same procedure to find $x_{k+1} \in {X_n}$ such that
$$
K_{x_{k+1}} \cap  \left( {\bigcup}_{j=1}^k K_{x_j}\right) \not =\emptyset \quad \hbox {and} \quad K_{x_{k+1}}\cap \left(C \setminus
{\bigcup}_{j=1}^kK_{x_j}\right) \not = \emptyset.
$$
Since $K = \bigcup_{x\in{X_n}} K_x$, this process must end at some step, say $\ell$, and in this case $C \subset {\bigcup}_{j=1}^\ell K_{x_j}$. Since the diameter $ r^{n+1}|K|< |K_{x_j}| \leq r^n |K| \to 0$ as $n \to \infty$, $\ell$ must tend to infinity, which contracts the uniform boundedness of the horizontal components on the levels $X_n$.

For the necessity, let $\{T_n\}_{n\ge 1}$ be a sequence of horizontal components satisfying $\lim_{n\to \infty}\#T_n=\infty$. Without loss of generality, we assume $T_n\subseteq X_n$.  Let $T$ be the set of limit points of $\{T_n\}_{n\ge 1}$. Then $T\subset \partial X$ and is a compact set. Since $\partial X$ is totally disconnected, there exists non-empty compact subsets $A, B\subset \partial X$ such that
$$A\cap B=\emptyset \quad \text{and}\quad A\cup B=T.$$
Let $0<\epsilon <\rho_a(A,B)$ such that
$$U=\{x\in X\cup\partial X: \rho_a(x, A)\le \frac \epsilon 3\}\quad\text{and}\quad V=\{y\in X\cup\partial X: \rho_a(y, B)\le \frac \epsilon 3\}$$
and $U\cap V=\emptyset$. Hence for $n$ large enough, $$U_n := U\cap T_n\ne\emptyset,\quad  V_n:= V\cap T_n\ne \emptyset.$$
Observe that $\lim_{n\to\infty} d(U_n, V_n)=\infty$ (here $d(U_n,V_n)=\inf\{d(u, v): u\in U_n, v\in V_n\}$). For otherwise, there exists $c_0>0$ and $\{n_k\}$ such that $$d(U_{n_k}, V_{n_k})\le c_0.$$ This implies that there exist $x_k\in U_{n_k}, y_k\in V_{n_k}$  such that $d(x_k, y_k)\le c_0$. Hence the limit points of $\{x_k\}$ and $\{y_k\}$  are identical in the hyperbolic boundary $\partial X$, a contradiction.

Let $D_n=T_n\setminus (U_n\cup V_n)$, by $\lim_{n\to \infty} d(U_n, V_n)=\infty$ and $T_n$ is connected and $\lim_{n\to\infty} \#T_n=\infty$, we conclude that $\lim_{n\to\infty}\#D_n=\infty$ (similar argument as the above). Then by Lemma \ref{lem-tree3}, there exists $z_n\in D_n$ such that
$$d(z_n, U_n)\ge \frac 13d(U_n, V_n), \quad d(z_n, V_n)\ge \frac 13d(U_n, V_n).$$  Let $z\in \partial X$ be a limit point of $\{z_n\}$, then $z\notin U\cup V (\supseteq A\cup B=T)$, a contradiction.  This completes the proof of the first part. The similar argument can also be used to show the second part without making many changes.
\end{proof}

\section{Concluding remarks and open questions}

In \cite{LaLu13}, we observed that if the incidence matrix $A$ of an augmented tree is primitive, then $A^k$ is rearrangeable for some $k>0$, by using this, we obtained a stronger but simpler result (Theorem 1.1 of \cite{LaLu13}). However, we do not know if the observation is still true for the generalized rearrangeable matrix of this paper. So we ask

\textbf{Q1.} Let $A, B, \bu$ be defined as in Definition \ref{def rearrangeable} and Remark \ref{B-u-remark}.  If $A, B$ are primitive matrices, does it imply that $A$ is $(B, \bu)$-rearrangeable?

Later in  \cite{DLL15}, for the equicontractive IFS with the OSC, by some modification, we removed the primitive (or rearrangeable) assumption on the incidence matrix.  So for the the IFS with non-equal contraction ratios and overlaps, we wonder

\textbf{Q2.} If we can remove  or weaken the rearrangeable assumption in our main results?

On the other hand,  a more general augmented tree was introduced by Lau and Wang \cite{LaWa16}, which is defined on a tree (independent of the IFS) with an associated set-valued map. They proved that such augmented tree is hyperbolic and  any compact set in ${\mathbb R}^d$ can be H\"older equivalent to the hyperbolic boundary of certain augmented tree. Hence more flexible iterated schemes can also be fit into the framework.  As for the application to Lipschitz equivalence, by using the method of this paper, it is still interesting to carry out a similar study on Moran sets or more general fractals.

\bigskip
\noindent {\bf Acknowledgements:} The author gratefully acknowledges the support of K. C. Wong Education Foundation and DAAD. He also would like to thank  Professor Ka-Sing Lau for his support and valuable discussions, especially on Section 5.

\end{document}